\documentclass{article}
\usepackage{amsmath, amsfonts, amsthm, amssymb, verbatim,dsfont}
\usepackage{graphicx}
\usepackage[mathcal]{euscript}
\usepackage{dsfont} 
\usepackage{amssymb,mathrsfs}
\usepackage[usenames]{color}
\theoremstyle{plain}
        \newtheorem{theorem}{Theorem}[section]
        \newtheorem{proposition}[theorem]{Proposition}
        \newtheorem{lemma}[theorem]{Lemma}
        
\theoremstyle{definition}
        \newtheorem{definition}[theorem]{Definition}
        
\theoremstyle{plain}
        
\numberwithin{equation}{section}

\usepackage{amsmath}
\usepackage{verbatim}
\newcommand \be           {\begin{equation}}
\newcommand \ee            {\end{equation}}
\newcommand \Dcal           {\mathcal{D}}

\newcommand \RR           {\mathbb{R}}

\newcommand \ZZ           {\mathbb{Z}}
\newcommand \CC           {\mathbb{C}}

\newcommand \Pbold           {\mathbf{P}} 
\newcommand \PP \Pbold

\newcommand \del           \partial
\newcommand \eps            \epsilon

\newcommand \OO     {\mathcal{O}}

\newcommand \loc        {{\mathrm{loc}}}
\newcommand{\unj}{u^n_j}

\newcommand{\unnj}{u^{n+1}_j}
\newcommand{\unnjj}{u^{n+1}_{j+1}}
\newcommand{\unnjjj}{u^{n+1}_{j+2}}
\newcommand{\unnjm}{u^{n+1}_{j-1}}
\newcommand{\unnjmm}{u^{n+1}_{j-2}}
\newcommand{\uh}{u^h}

\newcommand{\uj}{u_j}
\newcommand{\ujm}{u_{j-1}}
\newcommand{\ujj}{u_{j+1}}

\def\XXint#1#2#3{{\setbox0=\hbox{$#1{#2#3}{\int}$}
\vcenter{\hbox{$#2#3$}}\kern-.5\wd0}}

\let\oldmarginpar\marginpar
\renewcommand\marginpar[1]{\-\oldmarginpar[\raggedleft\footnotesize #1]%
{\raggedright\footnotesize #1}}
\def\build#1_#2^#3{\mathrel{
\mathop{\kern 0pt#1}\limits_{#2}^{#3}}}
\allowdisplaybreaks[1]

\begin{document}

\title{Convergence of a finite difference method for the KdV and modified KdV equations with
$L^2$ data}
%

\author{
{ Paulo Amorim\thanks{Corresponding author. Email: pamorim@ptmat.fc.ul.pt}\,\, and M\'ario Figueira\thanks{Email: figueira@ptmat.fc.ul.pt} }\\[5pt]
Centro de Matem\'atica e Aplica\c c\~oes
Fundamentais,\\
Departamento de Matem\'atica,\\
Universidade de Lisboa, Av. Prof. Gama Pinto 2,\\
1649-003 Lisboa, Portugal. 
}

\maketitle

\begin{abstract} 
{We prove strong convergence of a semi-discrete finite difference method for the KdV and modified KdV 
equations. We extend existing results to non-smooth data (namely, in $L^2$), without size restrictions.
Our approach uses a fourth order (in space) stabilization term and a special conservative 
discretization of the nonlinear term. Convergence follows from a smoothing effect and energy estimates.
We illustrate our results with numerical experiments, including a numerical
investigation of an open problem related to uniqueness posed by Y.~Tsutsumi.
}\\[5pt]
\end{abstract}
 

%
%
\newline

\section{Introduction}
This paper is concerned with the study of a numerical approximation of the equation
\be
\label{KdV0}
\del_t u + \del_x^3 u + \beta  \del_x u^{k+1} = 0, \qquad \beta \neq 0, \quad k =1,2,
\ee
with $u = u(x,t),$ $x \in \RR$, $t\ge 0$. When $k=1$, the equation \eqref{KdV0} is referred to as 
the Korteweg--de Vries (KdV) equation, and when $k=2$ as the modified KdV  (mKdV) equation.

As is well-known, the KdV equation describes the unidirectional propagation of small-but-finite
amplitude waves in a nonlinear dispersive medium. It appears in several physical contexts,
such as shallow water waves and ion-acoustic waves in a cold plasma. Also, the modified
KdV equation has been used to describe acoustic waves and Alf\'en waves in plasmas without collisions.
For a more complete description of the physical contexts concerning the Korteweg--de Vries equation
and its generalizations, see \cite{Scott} and the references therein.

A large amount of work on the KdV equation was initially directed toward the study of solitary waves,
i.e., solutions of the form $u(x,t) = U(x-ct)$, especially the so-called soliton solutions, a class of 
solitary waves which preserve the form through nonlinear interaction (see \cite{Scott, Miura} for 
surveys on solitons). 
One of the most relevant results in soliton theory was the development of the 
inverse scattering method,
initially applied to the KdV equation by Gardner et.~al.~\cite{MiuraEtc} and, in a general form,
by Lax \cite{Lax}. This technique was also used to obtain solutions of the KdV equation
with low regularity \cite{Cohen1,Cohen2,Cohen3}.

Here, we concentrate on the numerical approximation of the solution of the Cauchy problem

\begin{subequations}
\label{KdV}
\begin{align}
\label{KdV1}
\del_t u + \del_x^3 u + \beta  \del_x u^{k+1} &= 0, \qquad \beta \neq 0, \quad k =1,2,
\\
u(x,0) &= \varphi(x), \qquad \varphi \in L^2.
\label{KdV2}
\end{align}
\end{subequations}
The mathematical problem of well-posedness for \eqref{KdV1},\eqref{KdV2} has been extensively studied.
We refer to the pioneering results in \cite{BonaScott,BonaSmith,SautTemam} and the improvements in
\cite{Kato1,Kato2}. In these works,
local well-posedness is proved in the Sobolev spaces $H^s$, $s>3/2,$ for generalized KdV (gKdV), 
in which the term $\del_x u^{k+1}$ is
replaced by $\del_x V(u)$.

Existence and uniqueness was also obtained in \cite{GT,GTV} with initial data in weighed $L^2$ and
$H^1$ spaces. In our numerical approach, we follow the energy method used in those papers.

More recently, following the introduction by Bourgain \cite{Bourgain} of certain Fourier spaces, the 
well-posedness result is strongly improved for data in negative Sobolev spaces (see the
monograph \cite{LinaresPonce} and the references therein), and uniqueness of solution in $L^2$
is proved in \cite{Zhou}.

Regarding the numerical solution of the KdV equation, convergence results have been proven for
a linearized equation \cite{Goda} and for smooth solutions \cite{YuSanzSerna}. However (to our 
knowledge), the problem of proving rigorous convergence of numerical
schemes without smoothness assumptions has only attracted attention in more recent years. 
Nixon \cite{Nixon} proves the convergence of 
approximate solutions for a discretized version of 
gKdV, but for small $L^2$ initial data, only. That work is the numerical counterpart of \cite{KPV}.
Finally we refer to the recent work by Pazoto et.~al.~\cite{Pazoto}, dealing with the 
numerical treatment
of the mKdV equation with critical exponent and a damping term, which shares some techniques with
the present work.

Thus, to the best of our knowledge, the problem of rigorous convergence of numerical
schemes for the KdV and mKdV equations with general data in $L^2$ has remained unsolved. 
The purpose of this paper
is to fill that gap.

Although the techniques we use to prove our convergence result are based on the ones
in \cite{GTV} (namely, the use of a fourth order stabilization term), their application to the numerical case
is not trivial. Indeed, it is essential to use a special non-conservative discretization of the nonlinear term 
in \eqref{KdV0}. This idea dates back at least to \cite{Goda}, and is also used in \cite{Pazoto}. Moreover, to obtain the necessary estimates for the
numerical approximation, additional technical difficulties related to interpolators are encountered,
with which we deal below.

An outline of the paper follows. After some notations and definitions in Section~\ref{Sec01}, we
prove our main convergence result in Section~\ref{Conv}.
In Section~\ref{Num} we present some numerical experiments to illustrate our convergence
results and test the accuracy of our scheme. 

Finally, in Section~\ref{Miura}, we investigate
numerically an open question posed by Y.~Tsutsumi \cite{Tsutsumi} relating to the uniqueness
of solution to the Cauchy problem for the KdV equation with measure initial data. 
This is done by means of the Miura transformation (see \cite{Tsutsumi}), 
which relates solutions of the KdV equation
with measure initial data to solutions of the mKdV equation with $L^2$ initial data. As explained 
in more detail in Section~\ref{Miura}, the
numerical evidence we provide suggests that the Cauchy problem for the KdV equation
with measure initial data is ill-posed. Note that, importantly, these numerical simulations 
involve discontinuous initial data in $L^2$ only, and, as such, are not covered by previous 
convergence results.

\section{Notations and definitions}
\label{Sec01}
Let $h$ denote our discretization parameter. We denote by $\uh_j = \uh_j (t)$ the (semi-discrete) 
difference approximation of $u(x_j, t)$, $x_j = jh$, $j\in\ZZ$. For $h>0,$ we define the Banach spaces
\[
\aligned
\ell^p_h (\ZZ) = \big\{ (z_j): z_j \in \CC, \| z\|_{p,h}^p \equiv \sum_{j\in\ZZ}h |z_j|^p < \infty \big\}.
\endaligned
\]
For $p=2$, we denote the usual scalar product by
\[
\aligned
(z,w)_h = \sum_{j\in\ZZ}h \,z_j \bar w_j,
\endaligned
\]
$z = (z_j),$  $w = (w_j)$. Let us also introduce the following standard notations for finite difference
operators. For $u = (u_j),$
\[ D_+ \uj = \frac1h(\ujj - \uj), \qquad D_- \uj = \frac1h( \uj - \ujm), \]
\[ D_0 \uj  = \frac1{2h} ( \ujj - \ujm) =  \frac12(D_+ + D_-) \uj,\]
\[ \Delta_h \uj = D_+ D_- \uj = D_- D_+ \uj = \frac1{h^2} (\ujj -2 \uj + \ujm),  \]
Also, denote the translation operators by 
\[ (u_+)_j = \ujj, \qquad (u_-)_j = \ujm. \]
We obtain the following formulas for the discrete differentiation of a product,
\begin{subequations}\label{DD}
\be
\label{2-10}
 D_+( vu) = v D_+u + u_+ D_+v
\ee
\be
\label{2-20}
 D_-( vu) = v D_-u + u_- D_-v 
\ee
\be
\label{2-30}
 D_0( vu) = v D_0u + \frac12 \big( u_+ D_+v  + u_- D_-v \big)
\ee
\be
\label{2-40}
\Delta_h(vu) = v \Delta_h u + u \Delta_h v + D_+v D_+u + D_- v D_-u.
\ee
\end{subequations}
Also, the difference operators verify in $\ell^2_h$
\[ (D_+ u, v)_h = - (u , D_- v)_h, \qquad (D_0 u, v)_h = - (u , D_0 v)_h,\]
and so
\[ (\Delta_h u, v)_h = (u , \Delta_h v)_h. \] 
We will also need to denote for a sequence $(u_j)$ and for a function $w$
\[
\| u\|_{p,R,h}^p = \sum_{|jh| \le R} h |u_j|^p, \qquad R>0,
\]
\[
\| w\|_{p,R} = \| w \|_{L^p(-R,R)}, \qquad 0< R \le +\infty.
\]
Finally, we introduce the continuous piecewise linear interpolator
\be
\label{defP1}
\aligned
P_1^h u(x) = \uj + (x-x_j) \frac{\ujj - \uj}{x_{j+1} - x_j}, \qquad x\in (x_j, x_{j+1}),
\endaligned
\ee
and the piecewise constant interpolator $ (P_0^h u) (x) =\uj,$ $x\in (x_j, x_{j+1}).$

\section{Convergence results }
\label{Conv}
In this section, we prove our main result, Theorem \ref{Thm-01}, which establishes the convergence of a
numerical approximation of problem \eqref{KdV}.

Let us consider the semi-discrete finite difference scheme
\begin{subequations}
\label{FDA}
\begin{align}
\label{FDA1}
&\frac{d}{dt} \uh +  D^3 \uh  + \beta\frac{k+1}{k+2} [ (\uh)^k D_0 \uh 
+ D_0 (\uh)^{k+1} ] + h \Delta_h^2 \uh = 0,
\\
\label{FDA2}
&\uh(0) = \varphi^h,
\end{align} 
\end{subequations}
where $ D^3 \uj = D_+ D_0 D_- \uj$ and $\Delta^2_h = D_+D_-D_+D_-$ denotes the difference
bi-laplacian, and $\uh$ denotes the unknown grid function $(\uh_j)_{j \in \ZZ}$, $\uh_j(t) $
being the approximation of the solution of \eqref{KdV} at the point $(x_j, t)$. 

The term 
$h \Delta_h^2 \uh$ is introduced in our scheme in order to obtain the uniform (in $h$) stability 
estimates necessary for the convergence proof. This term corresponds to the parabolic regularization
used in \cite{GT} for the continuous problem. Also, the formally consistent discretization 
\be
\label{disc}
\beta  \del_x u^{k+1} \sim \beta\frac{k+1}{k+2} [ (\uh)^k D_0 \uh 
+ D_0 (\uh)^{k+1} ] 
\ee
is based on a corresponding one in \cite{Goda} and is also essential in our proof. See also 
\cite{Pazoto} for an application of the same idea in a different setting.

The following first result holds.

\begin{proposition}
\label{Prop-01}
Let $h>0$. Then, for each initial data $\varphi^h \in \ell^2_h(\ZZ)$, there exists a unique global solution
$\uh(t)\in C(\RR; \ell^2_h(\ZZ))$ of \eqref{FDA}.
\end{proposition}
\begin{proof}
Existence of a unique local solution in $\ell^2_h(\ZZ)$ follows from the Banach fixed-point theorem. The 
global existence is an immediate consequence of the uniform bounds on the $\ell^2_h$ norm,
established below in Lemma \ref{lem-L2}.
\end{proof}

Let $P_1^h$ denote the continuous piecewise linear interpolator and let $\varphi \in L^2(\RR)$ be the
initial data for the problem \eqref{KdV}. Also, we denote by $C_w(I; X)$ the space of weakly continuous
functions from the interval $I$ to the Banach space $X$. 

It is now convenient to introduce the notion of weak solution to the Cauchy problem
\eqref{KdV1},\eqref{KdV2}.

\begin{definition}
\label{def-sol}
Let $\varphi \in L^2(\RR)$. A function $u(x,t)$, $x\in\RR$, $t\ge0$, is a solution to the Cauchy problem \eqref{KdV1},\eqref{KdV2}
if 
\enumerate
\item $u \in L^\infty_\loc((0,\infty); L^2(\RR))$,
\item For each test function $\phi \in \Dcal(\RR\times(0,\infty))$ one has
\be
\label{sol}
\aligned
\int_{\RR^2} u( \del_t \phi + \del^3_x \phi) + \beta u^{k+1} \del_x \phi \,dx \,dt = 0,
\endaligned
\ee
\item $u(t) \to \varphi$ in $L^2(\RR)$ as $t \to 0$ a.e.
\endenumerate
\end{definition}

We now state the main result of this paper.

\begin{theorem}[Convergence of approximate solutions]
\label{Thm-01}
Let $\varphi^h \in \ell^2_h(\ZZ)$ be the initial data for the discretized problem 
\eqref{FDA}, such that $ P_1^h 
\varphi^h \to \varphi $ in $L^2(\RR)$ when $h \to 0$. Then, for each $T>0$, the sequence 
$ P_1^h \uh$ satisfies
\begin{align*}
&{ P}_1^h u^h \rightharpoonup  u\quad \text{in} \quad L^2([0, T] ; H^1_\loc(\RR))
 \quad \text{ weak},
   \\
&{ P}_1^h u^h  \rightarrow u\quad \text{in} \quad L^2([0, T] ; L^2_{\loc}(\RR))    
\end{align*}
with $u$ a solution of \eqref{KdV}. Moreover, $u$ satisfies
\be
\label{thm-10}
 u \in (L^\infty \cap C_w) ( [0,T]; L^2_\loc (\RR)),
\ee
\[ 
\| u(t) \|_2 \le \| \varphi\|_2 \quad \text{ and } \quad \| u(t) - \varphi \|_2 \to 0
\quad \text{ as } h\to 0.
\]
\end{theorem}

The proof will be postponed to the next section.

\subsection{Main estimates}

First, let us record some inequalities which will be of use throughout. Let 
$v = (v_j)\in \ell^q_h(\ZZ),$ $q\in[1,\infty]$. From \eqref{defP1}
we derive
\be
\label{P1P0}
\aligned
&\| P_1^h v - P_0^h v \|_{q,R} \le C h \| D_0 v\|_{q,R,h}, \qquad 0<R\le \infty,
\\
&\| P_1^h v - P_0^h v \|_{q,R} \le C  \|  v\|_{q,R,h}, \qquad 0<R\le \infty,
\quad q \in [1, +\infty],
\endaligned
\ee
for some $C$ independent of $h$. As a consequence, we obtain
\be
\label{P1}
\| P_1^h v\|_{q,R} \le C\| P_0^h v\|_{q,R} = C \| v\|_{q,R,h}, \forall v\in \ell^q_h(\ZZ), \quad0<R\le +\infty.
\ee
We will need the following inequalities,
\begin{lemma}
Let $\phi \in \ell^2_h(\ZZ)$. Then,
\be
\label{Ineq01}
\| \phi \|_\infty \le C \| \phi\|_{2,h}^{1/2}\| D_\pm\phi\|_{2,h}^{1/2}
\ee
\be
\label{Ineq02}
\| \phi \|_\infty \le \frac12 \| D_\pm\phi\|_{1,h}
\ee
\end{lemma}
\begin{proof}
The inequality \eqref{Ineq01} is a consequence of the Gagliardo-Niremberg inequality and $\del_x P_1^h
= P_0^h D_+$:
\[
\aligned
\| \phi\|_\infty &= \| P_1^h\phi\|_\infty \le C \| P_1^h\phi\|_{2}^{1/2}
\| \nabla P_1^h\phi\|_{2}^{1/2} 
\\
& \le C \| \phi\|_{2,h}^{1/2}
\| D_\pm\phi\|_{2,h}^{1/2} ,
\endaligned
\]
while \eqref{Ineq02} is a consequence of the (continuous) inequality $\| \phi\|_\infty \le \frac12
\| \phi'\|_1$.
\end{proof}

We are now ready to state our first stability estimate.
\begin{lemma}
\label{lem-L2}
Let $\uh(t)$ be a solution of \eqref{FDA}. Then, for all $t>0$, it holds
\be
\label{L2}
\| \uh(t) \|_{2,h} \le \| \varphi^h\|_{2,h}.
\ee
In particular, $\uh$ is a global solution of \eqref{FDA}.
\end{lemma}
\begin{proof}
In the next proofs, we will, for simplicity, omit $h$ from the notation.
Take the $\ell^2$-scalar product of the equation \eqref{FDA1} with $u \equiv \uh$ to get
\[
\aligned
( \frac d{dt}u , u ) +& \beta\frac{k+1}{k+2} \big[ (u^k D_0 u , u ) + ( D_0 u^{k+1}, u) \big]
\\
&+ h (\Delta_h^2 u , u ) =0.
\endaligned
\]
Now, 
\[ 
(\Delta_h^2 u , u ) = (\Delta_h u , \Delta_h u ) = \|\Delta_h u\|_2^2,
\]
\[ (u^k D_0 u , u ) =  (u^{k+1}, D_0 u  ) = - ( D_0 u^{k+1}, u), \]
and so 
\be
\label{L21}\frac d{dt} \| u(t) \|_2^2 = - h\|\Delta_h u\|_2^2 \le 0,
\ee
from which the conclusion follows.
Notice how the discretization \eqref{disc} leads to the non-increase of the $\ell^2$ norm.
\end{proof}

The next lemma is a fundamental identity which, as we will see in Proposition~\ref{prop-smoot} below,
implies a smoothing effect inherent to the equation: even though the initial data is only in
$\ell^2_h(\ZZ)$, the solution of \eqref{FDA} is actually in a more regular space (uniformly in $h$),
namely, $H^1_\loc$.

Let $p: \RR \to \RR$ be a bounded, strictly increasing, smooth function, with all its derivatives bounded.
Write $p_j = p(x_j),$ $j\in \ZZ$. For simplicity, we do not distinguish in our notation 
the continuous and the discrete $p$.
\begin{lemma}
\label{lem-p}
Let $\uh = \uh(t)$ be the solution of the discrete problem \eqref{FDA}. Then, $\uh$ satisfies the identity
\be
\label{Id-p}
\aligned
&\frac12 \frac d{dt} \| p^{1/2} \uh \|_2^2+ (D_- \uh, D_+ p D_- \uh) 
+ \frac12 (D_+ \uh, D_- p D_+ \uh) 
\\
& + h (D_+ D_- \uh, p D_+ D_- \uh) =  -\frac h2 (D_+ D_- \uh,  D_+ p D_- \uh) 
\\
& + \frac h2 (D_+\uh D_- p,  D_+ D_- \uh) -   (D_- \uh,  \uh_-D_0 D_- p ) 
\\
& - h (D_+ D_- \uh,  D_- p D_- \uh)  - h (D_+ D_- \uh,  D_+ p D_+ \uh) 
\\
&  - h (D_+ D_- \uh,  \uh D_+ D_- p ) 
\\
& +\frac\beta2\frac{k+1}{k+2} \big( (\uh)^{k+1}, \uh_+ D_+p + \uh_- D_- p \big).
\endaligned
\ee
\end{lemma}
\begin{proof}
We take \eqref{FDA1}, multiply by $h p_j \uh_j$, and sum over $j\in\ZZ$ to obtain
\be
\label{lem-p01}
\frac12 \frac d{dt} \| p^{1/2}u\|_2^2 + (D^3 u, pu) + \beta\frac{k+1}{k+2} \big[ (u^k D_0 u , p u ) 
+ ( D_0 u^{k+1}, p u) \big] + h (\Delta_h^2 u , p u ) = 0.
\ee
We find from \eqref{DD} 
\[
\aligned
(D^3 u, pu) &= ( D_+D_0D_- u, pu) = -(D_0D_- u, D_- (pu))
\\
&= - (D_0D_- u, p D_-u) - (D_0D_- u, u_- D_- p) = A + B.
\endaligned
\]
Since
\[
\aligned
A = (D_-u, p D_0D_- u) + \big( D_-u, \frac12 (D_+p (D_-u)_+ + D_-p (D_-u)_-) \big),
\endaligned
\]
we obtain 
\[
2A = (D_- u, \frac12 D_+p D_+u) + (D_-u, \frac12 D_-p (D_-u)_- )
\]
and so 
\[
\aligned
A &= \frac14 \big[ (D_-u, D_+p D_+u ) + ( D_+u, D_+p D_-u) \big]
\\
& = \frac12 (D_+u, D_+p D_-u)
\\
&= \frac12 \big[ (D_+u - D_-u, D_+p D_-u) + (D_-u, D_+p D_-u) \big]
\\
&= \frac{h}2  (D_+ D_-u, D_+p D_-u) + \frac12 (D_-u, D_+p D_-u) .
\endaligned
\]
Similarly,
\[
\aligned
B &= (D_-u, u_- D_0D_-p ) + \big( D_- u, \frac12[ D_+u_- D_+p + D_-u_-(D_-p)_- ] \big)
\\
& = (D_-u,  u_- D_0 D_-p) + \frac12 (D_-u, D_+p D_-u) + \frac12(D_+u, D_-u D_-p).
\endaligned
\]
Since 
\[ (D_+ u, D_-p D_-u) = (D_+u, D_-pD_+u) - h(D_+u, D_-p D_+D_- u), \]
we obtain
\[
\aligned
B & = (D_-u, u_- D_0D_-p ) + \frac12 (D_-u, D_+pD_-u)
\\
&+ \frac12 (D_+u, D_-p D_+u) - \frac{h}2 (D_+u, D_-p D_+D_- u).
\endaligned
\]
For the term in \eqref{lem-p01} corresponding to the discrete bi-laplacian, we derive
\[
\aligned
(\Delta_h^2 u , p u ) &=  - (D_-D_+D_- u, p_- D_-u + uD_-p) 
\\
&= -(D_-D_+D_- u, p_- D_-u) - (D_-D_+D_- u, uD_-p)  
\\
& = (D_+D_- u, p D_+ D_-u) + (D_+D_- u, D_-p D_-u )
\\
&\quad + (D_+D_- u,  D_+p D_+u) + (D_+D_- u, u D_+ D_-p  ).
\endaligned
\]
As to the remaining term in \eqref{lem-p01}, we find
\[
\aligned
(u^kD_0 u, pu) + (D_0 u^{k+1}, pu) &= (D_0 u, pu^{k+1}) - ( u^{k+1}, D_0(pu))
\\
&= - \big( u^{k+1}, \frac12 (u_+ D_+p + u_- D_-p) \big).
\endaligned
\]
All these results together give \eqref{Id-p}. This completes the proof of Lemma \ref{lem-p}.
\end{proof}

As a consequence of the two preceding lemmas, we now prove the following result which states that,
at the discrete level, $\uh$ is in $H^1_\loc$.
\begin{proposition}
\label{prop-smoot} 
Let $\uh$ be solution of the discretized problem \eqref{FDA} with initial data $\varphi = \varphi^h
\in \ell^2_h(\ZZ)$. Then, for each $T>0$ and for each $R>0$, there exists a constant $C
=C(R,T, \| \varphi\|_{2,h})$ such that, for all $h>0$,
\be
\label{Estim}
\int_0^T \sum_{|jh| \le R} h | D_\pm u^h_j|^2 \,dt \le C.
\ee
\end{proposition}
\begin{proof}
We apply Lemma \ref{lem-p} with a bounded, strictly increasing, smooth function $p$, with all its derivatives bounded, and such that, moreover, $p(x)\ge1$ for all $x$, and $p'(x) =1$ for 
$x\in [-R,R]$. Let us rewrite the identity \eqref{Id-p}, with obvious notation,  as
\be
\label{Id-p10}
\aligned
&\frac12 \frac d{dt} \| p^{1/2} \uh \|_2^2+ (D_- \uh, D_+ p D_- \uh) 
+ \frac12 (D_+ \uh, D_- p D_+ \uh) 
\\
& + h (D_+ D_- \uh, p D_+ D_- \uh) = A_1 + \cdots + A_7.
\endaligned
\ee
Now observe that under our assumptions on $p$, the terms on the left-hand side (except the first)
are non-negative,
so we must bound the terms $A_i$.

The terms $A_1$ and $A_2$ are similar and yield
\[
\aligned
\frac h2 |(D_+ D_- \uh,  D_+ p D_- \uh) | &\le \frac h2 \|(D_+ D_- \uh) p^{1/2}\|_2 \,
\| p^{-1/2}D_+p D_- \uh\|_2
\\
& \le \eta \frac h2 \|(D_+ D_- \uh) p^{1/2}\|_2^2 + \frac{Ch}{2\eta}\| (D_+ p)^{1/2} D_- \uh \|_2^2,
\endaligned
\]
for all $\eta>0$, where we have used the properties of $p$.
Also, the terms $A_4$ and $A_5$ are similar and give
\[
\aligned
&h (D_+ D_- \uh,  D_- p D_- \uh) 
\\
&=  h (D_+ D_- \uh,  D_- p D_+ \uh)
-  h^2 (D_+ D_- \uh,  D_- p D_+D_- \uh) = B_1 + B_2.
\endaligned
\]
The term $B_1$ is similar to $A_1$, while
\[
\aligned
|B_2| \le C h^2 ( D_+ D_- \uh, p D_+ D_- \uh).
\endaligned
\]
For the term $A_3$, we remark that
\[
\aligned
&(D_- \uh,  \uh_-D_0 D_- p ) =  (D_- \uh - D_-\uh_-,  \uh_-D_0 D_- p ) 
+  (D_- \uh_-,  \uh_-D_0 D_- p )
\\
& = h(D_-D_- \uh,  \uh_-D_0 D_- p ) - (\uh_-,  D_-\uh D_0 D_- p ) 
- ( \uh_-,  \uh D_+D_0 D_- p )
\endaligned
\]
and so
\[
\aligned
|(D_- \uh,  \uh_-D_0 D_- p )| &\le Ch \| p^{1/2} D_+D_- \uh \|_2 \|\uh\|_2 + C\| \uh\|_2^2
\\
&\le C_1h^2 \| p^{1/2} D_+D_- \uh \|^2_2 + C_2\| \uh\|_2^2.
\endaligned
\]
The term $A_6 =h (D_+ D_- \uh,  \uh D_+ D_- p ) $ is easily estimated using the 
Cauchy-Schwarz inequality. Finally, for the last term 
\[A_7 = \frac\beta2\frac{k+1}{k+2} \big( (\uh)^{k+1}, \uh_+ D_+p + \uh_- D_- p \big),\]
let us consider only $k=2$, since the case $k=1$ is easier. Setting $q = (D_+p)^{1/2}$ and $u=\uh$,
we obtain
\[
| (u^3, u_+ D_+ p) | \le C\| u\|_2^2 \|q u^2\|_\infty
\]
and from \eqref{Ineq02} it follows
\[
\aligned
\|q u^2\|_\infty  &\le \| (D_-q) u^2_- + q (u D_- u +u_- D_-u) \|_1
 \le C\| u\|_2^2 + \| qD_-u\|_2\|u\|_2.
\endaligned
\]
Hence,
\[
\aligned
| (u^3, u_+ D_+p)| &\le C \|u\|_2^4 + C \|u\|_2^3 \|(D_+p)^{1/2} D_-u \|_2
\\
& \le \eps \|(D_+p)^{1/2} D_-u \|_2^2 + \frac C\eps \|u\|_2^6 + C\| u\|_2^4,
\endaligned
\]
for any $\eps>0$.

Choosing $\eps, \eta$ small enough, and for some small enough $h_0$ (which does not depend on
$R,T$ or $\varphi$), we find, after integrating \eqref{Id-p10} on $[0,T]$ and using the previous estimates,
\[
\aligned
\int_0^T \sum_{|jh| \le R} h | D_\pm u^h_j|^2 \,dt \le C(R,T, \|\varphi\|_2).
\endaligned
\]
This completes the proof of Proposition \ref{prop-smoot}.
\end{proof}


\subsection{Proof of Theorem~\ref{Thm-01}}
The proof of Theorem \ref{Thm-01} relies on Aubin's compactness result,
which we state here, in a simplified form, for the reader's convenience.
\begin{lemma}[{\cite[p.~58]{Lions}}]
\label{lem-aubin}
Let $1<p<\infty$, $T>0$, and consider reflexive Banach spaces $B_0 \subset B \subset B_1$
such that $B_0$ is compactly embedded in $B$. Then, the space
\[
\aligned
\Big\{ v : v\in L^2\big( 0,T; B_0\big), \quad \frac{dv}{dt}  \in L^p\big( 0,T; B_1 \big) \Big\}
\endaligned
\]
is compactly embedded in $L^2\big( 0,T; B\big)$.
\end{lemma}
In order to apply Lemma \ref{lem-aubin}, we will use the following estimates.
\begin{lemma}
\label{lem-05}
Let $\uh$ be given by \eqref{FDA1},\eqref{FDA2} and let $T,R>0$. Then, there exists $p>1$ such that
\be
\label{proof-10}
\aligned
\int_0^T \big\| P_1^h \uh \big\|^2_{H^1(-R,R)} \,dt \le C, 
\endaligned
\ee
\be
\label{proof-20}
\aligned
\int_0^T \big\| \frac{d}{dt} P_1^h \uh \big\|^p_{H^{-5}(-R,R)} \,dt \le C,
\endaligned
\ee
uniformly in $h$, with $C = C(R, T, \| \varphi \|_{2,h})$.
\end{lemma}
\proof
First of all, note that since $\del_x P_1^h = P_0^h D_+$, it follows from Lemma~\ref{lem-L2} and Proposition~\ref{prop-smoot}
that the estimate \eqref{proof-10} holds for each $R>0$.

Let us now prove the estimate \eqref{proof-20}. Let $\uh$ be given by \eqref{FDA1},\eqref{FDA2}.
We apply the piecewise linear continuous interpolator $ P_1^h$ to 
the equation \eqref{FDA1} to obtain
\be
\begin{aligned}
\label{proof-30}
&\frac{d}{dt} P_1^h\uh +   P_1^h D^3\uh 
\\ 
&\qquad + \beta\frac{k+1}{k+2} \big( P_1^h [(\uh)^k D_0 \uh] 
+ P_1^h D_0 (\uh)^{k+1} \big) + h P_1^h\Delta_h^2 \uh = 0,
\end{aligned}
\ee
with $\uh = \uh(t)$. We begin by estimating the term $P_1^h D^3\uh $, for which it is convenient
to consider the decomposition $ P_1^h = (P_1^h - P_0^h) + P_0^h$ and analyze the two resulting terms.
For each test function $\phi \in \Dcal( -R,R)$ we have
\[
\aligned
\big( P_0^h D^3 \uh, \phi \big) &= \sum_{|jh| \le R}  \int_{x_j}^{x_{j+1}} P_0^h D^3 \uh \phi \,dx
\\
& = \sum_{|jh| \le R} D_+ D_0 D_- \uj \int_{x_j}^{x_{j+1}} \phi(x) \,dx
\\
& = \sum_{|jh| \le R} D_-\uj \int_{x_j}^{x_{j+1}} \frac1{h^2}( \phi(x-2h) - \phi(x-h) - \phi(x) 
+ \phi(x+h)) \,dx
\endaligned
\]
and so (by Taylor expansion of $\phi$)
\[
\aligned
\big|\big( P_0^h D^3 \uh, \phi \big)\big| &\le  C\sum_{|jh| \le R} h |D_-\uj| \| \phi''\|_\infty
\\
& \le C \Big( \sum_{|jh| \le R} h |D_-\uj|^2 \Big)^{1/2} \| \phi\|_{H^{3}(-R,R)}.
\endaligned
\]
Hence, by Proposition \ref{prop-smoot}, we have
\be
\label{proof-40}
\aligned
\int_0^T \big\| P_0^h D^3 \uh\big \|_{H^{-3}(-R,R)} \, dt \le C.
\endaligned
\ee
Next, if $x \in (x_j , x_{j+1})$ and $v_j = D^3 \uj$, we easily find 
\[ ( P_1^h - P_0^h) v^h (x) = (x- x_j)\frac{v_{j+1} - v_j}h, \]
and so, with obvious notation,
\[
\aligned
\big( (P_1^h - P_0^h ) D^3 \uh, \phi \big) &= \sum_{|jh| \le R}D_+ D^3
\uj \int_{x_j}^{x_{j+1}} (x-x_j)  \phi(x) \, dt 
\\
& = \frac1h \big( D_+ D^3 \uj , A_j\big)
\\
& = - \frac1h \big( D_- \uj ,  D_0 D_-D_- A_j\big).
\endaligned
\]
A straightforward computation gives
\[
\aligned
&D_0 D_-D_- A_j 
\\
&\qquad= \frac1{2h^3} \int_{x_j}^{x_{j+1}} \big( \phi(x+h) - 2 \phi(x) + 2 \phi(x-2h) - 
\phi(x-3h) \big) (x - x_j) \,dx
\endaligned
\]
from which we obtain by Taylor expansion of $\phi$ and Proposition~\ref{prop-smoot}
\[
\aligned
\int_0^T \big|\big( (P_1^h - P_0^h ) D^3 \uh, \phi \big)\big| &\le C \int_0^T
\sum_{|jh| \le R}h |D_- \uj| 
\|\phi'''\|_{\infty} \,dt
\\
& \le C \|\phi\|_{H^{4}(-R,R)} \int_0^T \Big(\sum_{|jh| \le R}h |D_- \uj|^2 \Big)^{1/2} \,dt
\\
& \le C \|\phi\|_{H^{4}(-R,R)}.
\endaligned
\]
From this and \eqref{proof-40} we obtain the estimate 
\be
\label{proof-50}
\aligned
\int_0^T \big\| P_1^h D^3 \uh\|^2_{H^{-4}(-R,R)} \, dt \le C.
\endaligned
\ee
In an entirely similar way, we arrive at
\be
\label{proof-60}
\aligned
\int_0^T \big\| P_1^h \Delta^2_h \uh\|^2_{H^{-5}(-R,R)} \, dt \le C.
\endaligned
\ee

It remains to estimate the nonlinear terms in \eqref{proof-30}. Choose a smooth function
$\theta : \RR \to \RR$ such that $\theta(x) = 1$ if $|x| \le R$ and $\theta(x) = 0$
if $|x| \ge R+1$. Using \eqref{P1P0} and \eqref{Ineq01} we derive, for $R>0$, $k = 1,2$,
\[
\aligned
\big\| P_0^h [ (\uh)^k D_0 \uh] \big\|_{3/2,R,h} & \le \| D_0\uh \|_{2,R+1,h} \| \theta\uh\|_{6k,h}^k
\\
&\le C \| D_0\uh \|_{2,R+1,h} \| \theta \uh\|_\infty^{k-1/3}
\\
&\le  \| D_0\uh \|_{2,R+1,h} \big( C + C  \| D_+ \uh\|_{2,K+1,h}^{k/2-1/6} \big)
\\
&\le C + C  \| D_+\uh \|_{2,R+1,h}^{k/2 + 5/6},
\endaligned
\]
with $C = C(\|\varphi\|_{2,h}, R)$. Choosing $p = 12/(3k+5) >1$, we obtain 
from Proposition~\ref{prop-smoot}
\[
\aligned
\int_0^T \big\| P_0^h [(\uh)^k D_0\uh] \big\|_{3/2}^p \,dt 
&\le C + C \int_0^T \big\| D_+ \uh \big\|_{2,R+1,h}^2 
\le C.
\endaligned
\]
Since  $\| (P_1^h - P_0^h)\uh \|_{3/2} \le C \| P_0^h \uh \|_{3/2}$ (cf.~\eqref{P1P0}), 
we conclude that
\be
\label{proof-70}
\aligned
\int_0^T \big\| P_1^h[(\uh)^k D_0 \uh ] \big\|_{3/2}^p \,dt \le C, \qquad k=1,2.
\endaligned
\ee

For the remaining nonlinear term $ P_1^h D_0 (\uh)^{k+1}$, we split it as above into
$(P_1^h - P_0^h) + P_0^h$. First, note that
\be
\label{proof-80}
\aligned
 \big\| P_0^h (\uh)^{k+1} \big\|_{2,R} &\le C \| \theta\uh\|^{k+1}_{2k+2} 
\le  C \| \theta\uh\|^{k}_{\infty} \le C + C\| D_+\uh\|^{k/2}_{2,R+1,h},
\endaligned
\ee
and, by \eqref{P1}, the same estimate is obtained for $\| P_1^h (\uh)^{k+1} \|_{2,R}$. Next,
since $ P_0^hD_0 = \del_x P_1^h$, we obtain for $k=1,2$
\[
\aligned
\int_0^T \big\| P_0^h D_0 (\uh)^{k+1}\big\|^{4/k}_{H^{-1}(-R,R)} \,dt & = 
\int_0^T \big\| \del_x P_1^h  (\uh)^{k+1}\big\|^{4/k}_{H^{-1}(-R,R)} \,dt 
\\
&\le \int_0^T \big\| P_1^h  (\uh)^{k+1}\big\|^{4/k}_{2,R} \,dt 
\\
&\le C + C \int_0^T \| D_+ \uh\|^2_{2,R+1,h} \,dt \le C.
\endaligned
\]
We now need to estimate $( P_1^h - P_0^h) D_0 (\uh)^{k+1}$. With computations similar to
the ones after \eqref{proof-40}, we find
\[
\aligned
\big| \big( ( P_0^h - P_1^h) D_0 (\uh)^{k+1}, \phi \big) \big| &\le
C \sum_{|jh| \le R} h |\uj|^{k+1} \| \phi''\|_\infty 
\\
&\le C \big\| P_0^h (\uh)^{k+1} \big\|_{2,R} \| \phi\|_{H^3(-R,R)},
\endaligned
\]
and by \eqref{proof-80},
\[
\aligned
\int_0^T \big\| (P_0^h - P_1^h) D_0 (\uh)^{k+1} \big\|^{4/k}_{H^{-3}(-R,R)} \,dt
\le C + C \int_0^T \| D_+ \uh\|^2_{2,R+1,h} \,dt \le C.
\endaligned
\]
Thus we conclude that
\be
\label{proof-90}
\aligned
\int_0^T \big\| P_1^h D_0 (\uh)^{k+1}\big\|^{4/k}_{H^{-3}(-R,R)} \,dt  \le C, \qquad k = 1,2.
\endaligned
\ee
The desired estimate \eqref{proof-20}, with $p=12/(3k+5) > 1$, now follows from  the estimates
\eqref{proof-30}, \eqref{proof-50}, \eqref{proof-60},
\eqref{proof-70}, and \eqref{proof-90}. This completes the proof of Lemma~\ref{lem-05}.
\endproof

\begin{proof}[Proof of Theorem \ref{Thm-01}]
In view of the estimates in Lemma~\ref{lem-05}, we apply Lem\-ma~\ref{lem-aubin} with
$p = 12/(3k +5)$, $B_0 = H^1(-R,R)$, $B = L^q(-R,R),$ $q\in (1, \infty]$, and $B_1= H^{-5}(-R,R)$ 
(note that $B_0\subset B$ with compact embedding). We conclude that, up to a subsequence,
$ P_1^h \uh$ converges weakly in $L^2 (0,T; H^1(-R,R))$ and strongly in $L^2(0,T; L^q(-R,R))$.
Using a diagonal argument, we obtain for a further subsequence 
\begin{align}
\label{proof-95}
&{ P}_1^h u^h \rightharpoonup  u\quad \text{in} \quad L^2([0, T] ; H^1(-R,R))
\quad \text{ weak},
   \\ \notag
&{ P}_1^h u^h  \rightarrow u\quad \text{in} \quad L^2([0, T] ; L^q(-R,R)), \quad R>0, 
\quad q \in (1,\infty],
\end{align}
for some $u \in L^\infty( 0,T; L^2(\RR)) \cap L^2(0,T; H^1_\loc(\RR))$, as $h\to 0$. 
Also, from \eqref{P1P0}
and Proposition~\ref{prop-smoot} we can conclude that
\be
\label{proof-100}
\aligned
P_0^h \uh \to u \quad \text{in}\quad L^2(0,T; L^q(-R,R)), \quad 1 < q \le 2.
\endaligned
\ee

Now we must prove that $u$ is a weak solution of the problem
\eqref{KdV1},\eqref{KdV2}, in the sense of Definition~\ref{def-sol}.
Let us apply the piecewise constant interpolator $ P_0^h$ to the discrete equation
\eqref{FDA1}:
\be
\label{proof-105}
\begin{aligned}
&\frac{d}{dt} P_0^h\uh +   P_0^h D^3\uh 
\\ 
&\qquad + \beta\frac{k+1}{k+2} \big( P_0^h [(\uh)^k D_0 \uh] 
+ P_0^h D_0 (\uh)^{k+1} \big) + h P_0^h\Delta_h^2 \uh = 0.
\end{aligned}
\ee
First, consider the linear terms. We take a test function $\phi \in \Dcal(\RR\times (0,\infty))$
and compute in the sense of distributions
\[
\aligned
\big( P_0^h D^3 \uh, \phi \big) &= \int_0^\infty \sum_{j\in\ZZ} D^3 \uj \int_{x_j}^{x_{j+1}}
\phi(x,t) \,dx \,dt
\\
& = - \int_0^\infty \sum_{j\in\ZZ} \uj \int_{x_j}^{x_{j+1}} \big( \del^3_x \phi + \OO(h) \big) \,dx \,dt
\\
& = -\big( P_0^h \uh , \del^3_x \phi \big) + \OO(h) 
 \to -\big( u, \del^3_x \phi \big) = \big( \del^3_x u, \phi \big)
\endaligned
\]
as $h\to 0$.
The term $h P_0^h\Delta_h^2 \uh $ is treated similarly and tends to zero as $h\to 0$ in the sense of
distributions.

Now consider the nonlinear terms. Note that $ P_0^hD_0 = \del_x P_1^h$ and write
$ P_0^h = (P_0^h - P_1^h) + P_1^h$.
Using \eqref{P1P0} and \eqref{Ineq01} we find
\[
\aligned
\big\| (P_1^h - P_0^h) (\uh)^{k+1} \big\|_{1,R,h} \le C h \|D_+ (\uh)^{k+1}\|_{1,R,h} 
\le Ch(1+ \|D_+ \uj\|^{3/2}_{2,R+1,h})
\endaligned
\]
and so
\be
\label{proof-110}
\aligned
(P_0^h - P_1^h) (\uh)^{k+1} \to 0 
\endaligned
\ee
in $L^{4/3}(0,T;L^1_\loc(\RR))$ as $h\to 0$. Since $ P_0^h$ commutes with the nonlinearity, 
it follows from \eqref{proof-100} that
\be
\label{proof-120}
\aligned
P_0^h (\uh)^{k+1} \to u^{k+1} \quad \text{in}\quad L^1(0,T;L^1_\loc(\RR)).
\endaligned
\ee
Hence, we deduce from \eqref{proof-110},\eqref{proof-120} that
\[
\aligned
P_0^h D_0(\uh)^{k+1} = \del_x P_1^h (\uh)^{k+1} \to \del_x u^{k+1} 
\endaligned
\]
in the sense of distributions. For the other nonlinear term, we note that 
$ P_0^h [(\uh)^k D_0 \uh] = P_0^h (\uh)^k P_0^h D_0 \uh$. We have
\[
\aligned
P_0^h (\uh)^k \to u^k \quad \text{in}\quad L^2(0,T; L^2_\loc(\RR))
\endaligned
\]
and, from \eqref{proof-95}, 
\[
\aligned
P_0^h D_0 \uh = \del_x P_1^h \uh \rightharpoonup \del_x u \quad \text{in}\quad
L^2(0,T; L^2_\loc(\RR)).
\endaligned
\]
Therefore, 
\[
\aligned
P_0^h [(\uh)^k D_0 \uh]  \to u^k \del_x u \quad \text{in} \quad L^1(0,T;L^1_\loc(\RR)).
\endaligned
\]
Multiplying \eqref{proof-105} by a test function in $\Dcal(\RR\times (0, \infty))$, the above convergences
allow us to conclude that $u$ verifies the property \eqref{sol} of Definition~\ref{def-sol}.

It remains to prove the weak $L^2(\RR)$-valued continuity property, \eqref{thm-10}, 
and that $u(t) \to \varphi$ in $L^2(\RR)$ as $t \to 0$ a.e. To prove the weak continuity property,
we remark that, for $\phi \in \Dcal(\RR)$, $t\in [0,T)$,
\[
\aligned
\big( P_1^h \uh(t + \rho) - P_1^h \uh(t), \phi \big) = \int_t^{t+\rho} \big( \frac{d \uh}{dt}
,\phi \big) \,ds
\endaligned
\]
and, from \eqref{proof-20}, we get
\[
\aligned
\big| \big( P_1^h \uh(t + \rho) - P_1^h \uh(t), \phi \big) \big| \le C\rho
\endaligned
\]
with $C= C(T,\phi)$, and so the family $t\mapsto P_1^h \uh(t)$ is uniformly bounded
in $L^2$ (see \eqref{P1})
and weakly equicontinuous. Therefore, the Ascoli--Arzel\`a Theorem implies that 
$u \in C_w([0,\infty); L^2(\RR))$.

Finally, since $\|u(t)\|_2 \le \| \varphi\|_2$ a.e. in $t$, we have
\[
\aligned
\| u(t) - \varphi\|_2^2 &\le \| u(t) \|_2^2 - 2 \big( u(t), \varphi \big) + \| \varphi\|_2^2 
\\
&\le 2 \| \varphi\|_2^2 - 2 \big( u(t), \varphi \big) \to 0
\endaligned
\]
for almost every $t\to 0^+$. This completes the proof of Theorem~\ref{Thm-01}.
\end{proof}


\section{Numerical experiments}
\label{Num}
\subsection{A fully discrete, fully implicit scheme}
In this section, we present some numerical experiments to test the accuracy of our scheme and to 
illustrate our results. In order to fully discretize the semi-discrete equations \eqref{FDA1}, we
use a fully implicit Euler scheme, as follows. Given a time step $\tau$ and a space step $h$, solve
for each $n =1,2,\dots$ the equations
\be
\label{num-10}
\aligned
&\frac{\unnj - \unj}\tau +  D^3 \unnj  + \frac{k+1}{k+2} [ (\unnj)^k D_0 \unnj
+ D_0 (\unnj)^{k+1} ] 
\\
&\qquad+ \eta\, h \Delta_h^2 \unnj = 0.
\endaligned
\ee
We have set $\beta = 1$ and introduced a new viscosity parameter $\eta>0$ allowing us to explicitly
control the amount of viscosity in the scheme.

As is standard in the numerical simulation of dispersive equations, we consider a sufficiently 
large spatial domain and initial data exponentially small outside some bounded region, 
ensuring that spurious wave reflection at the boundary of the domain remains negligible. 

Written in full, the scheme \eqref{num-10} reads
\[
\aligned
&\frac{\unnj - \unj}\tau +  \frac1{2h^3} \big( \unnjjj -2 \unnjj +2\unnjm -\unnjmm\big) 
\\
&\quad+ \frac{k+1}{2h(k+2)} \big( (\unnj)^k (\unnjj - \unnjm) +  (\unnjj)^{k+1} - (\unnjm)^{k+1} \big)
\\
&\qquad+ \frac\eta{2h^3} \big( \unnjjj - 4 \unnjj + 6 \unnj - 4\unnjm +\unnjmm\big) = 0.
\endaligned
\]
Due to the nonlinear terms, it is necessary to perform a Newton iteration at each time step,
which we carry out with a tolerance of $10^{-6}$ in the simulations below. To solve the 
pentadiagonal linear system at each iteration of Newton's method, we employ a standard $LU$ 
decomposition method.

\subsection{Comparison with exact solutions}
The first step is to test our scheme with the known soliton solutions of the KdV equation 
\eqref{KdV0} 
\cite[p.~140]{LinaresPonce}. These read 
\be
\label{exact}
\aligned
u(x,t) = \big[ \frac12(k+2) c^2 \mathrm{sech}^2(\frac k2c(x - c^2t)) \big]^{1/k}
\endaligned
\ee
for arbitrary $c>0$ and consist of traveling waves with speed $c^2$. We observe in passing that 
these exact solutions actually solve the equation \eqref{KdV0} and not the slightly different version
in \cite[p.~139]{LinaresPonce}.

In Figures \ref{fig-10} and \ref{fig-20}, we present the $L^2$ error between the 
exact solution \eqref{exact} and the 
computed solution, at $T=10$, computed on the domain $x\in(10,50)$, as a function of 
the number of spatial points, for different values of 
the time step $\tau$, and, respectively, for $k=1$ and $k=2$.

One advantage of the present method is that it allows direct control of the amount of dissipation
my means of the parameter $\eta$ in \eqref{num-10}. We first note that our convergence results 
remain valid for any $\eta>0$ (but not for $\eta = 0$). As would be expected, reducing the value of
$\eta$ provides a sharper, less dissipative approximation. This is confirmed by our simulations, and
in Figure~\ref{fig-30} we present the $L^2$ error at $T=10$ for various values of $\eta$.
Interestingly, setting $\eta=0$ sometimes provides a very good approximation, but not always,
which is perhaps a consequence of the instability of the scheme without dissipation.

\begin{figure}
\includegraphics[width=\linewidth,keepaspectratio=false]{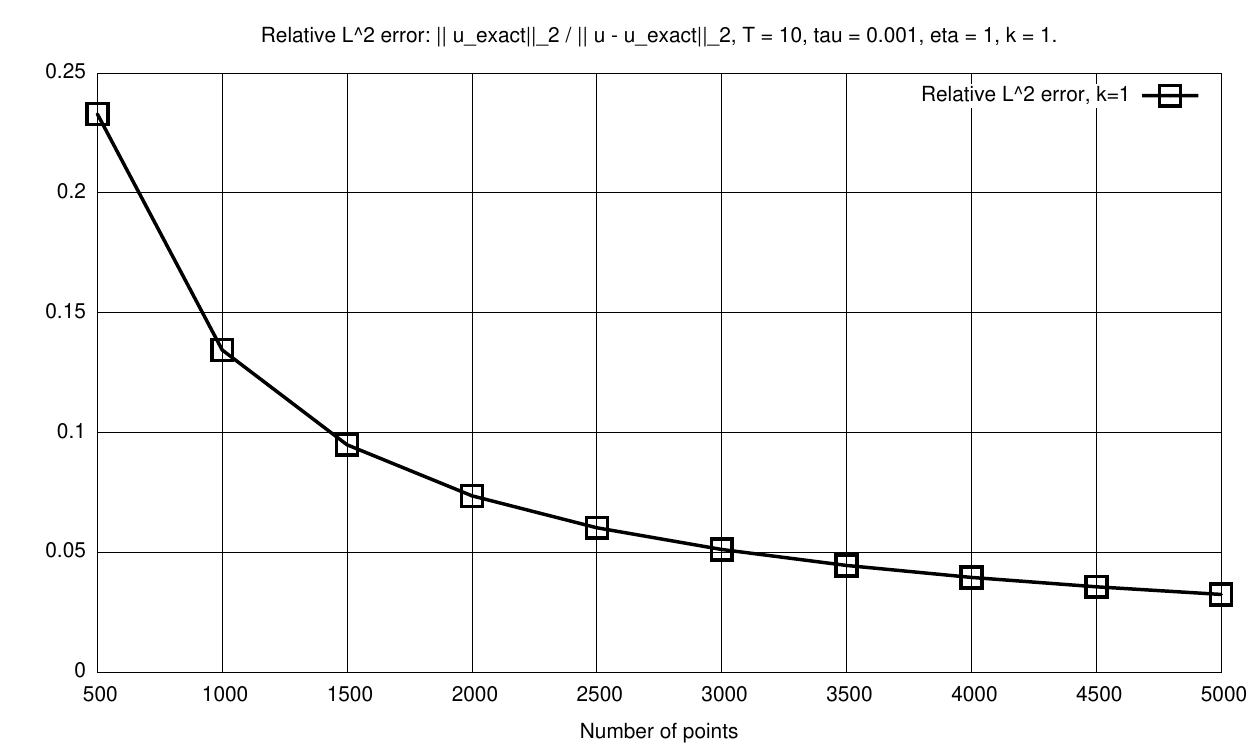}
\caption{Relative $L^2$ error as a function of the number of spatial points, k=1 (KdV equation).}
\label{fig-10}
\end{figure}

\begin{figure}
\includegraphics[width=\linewidth,keepaspectratio=false]{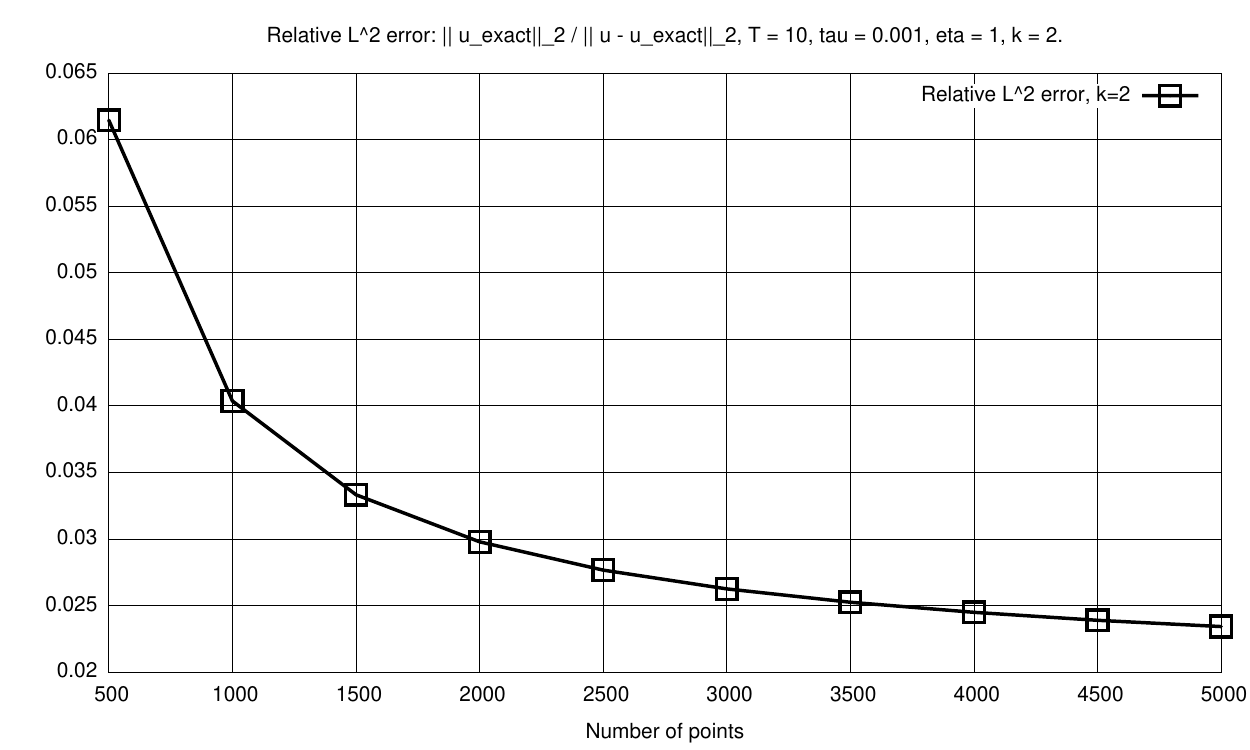}
\caption{Relative $L^2$ error as a function of the number of spatial points, k=2 (mKdV equation).}
\label{fig-20}
\end{figure}

\begin{figure}
\includegraphics[width=\linewidth,keepaspectratio=false]{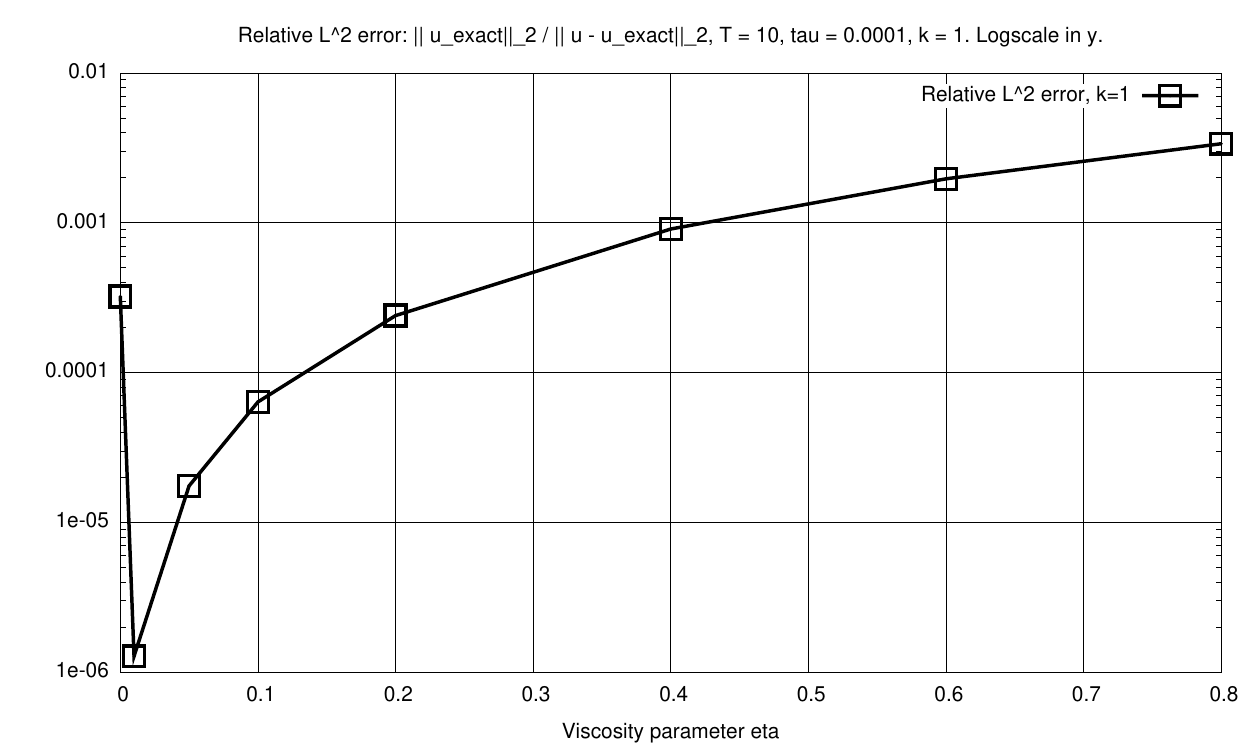}
\caption{Relative $L^2$ error as a function of the viscosity parameter $\eta$ (KdV equation).}
\label{fig-30}
\end{figure}

\section{On an open question of Y.~Tsutsumi}
\label{Miura}
In \cite{Tsutsumi}, the Cauchy problem for the KdV equation \eqref{KdV0} with measure initial
data is considered. In that work, the author addresses the open question of uniqueness
of solution to the Cauchy problem for the KdV equation with measure initial data 
in the following way.

It is well known that a solution of the mKdV equation
with $L^2$ initial data may be transformed, by the Miura transform
$u \mapsto \mathcal{M}(u) = \del_x u + u^2$, into a solution
of the KdV equation with a measure as initial data. 
Now, the family of functions
\be
\label{num-30}
u_c^0(x) = \left\{
\aligned
&\frac{c+1}{(c+1) x +c}, && x>0,
\\
&\frac1{x+c}, && x<0,
\endaligned
\right.
\ee
with $c\le -1,$ all verify $\mathcal{M}(u_c^0) = \delta(0)$, where $\delta$ denotes the Dirac delta.
Therefore, if $u_c(x,t)$ is the solution of the mKdV equation with initial data $u_c^0(x)$, the 
question arises whether the Miura transform maps each of these different solutions to the same
solution of the KdV equation with $\delta(0)$ as initial data, or if, on the contrary, 
$\mathcal{M}(u_c(x,t))$
varies with $c$, which would establish non-uniqueness. 
If the latter case is observed numerically, it
would support the conjecture 
that the Cauchy problem for the KdV equation with measure initial data 
\emph{does not} enjoy the uniqueness property.

We have investigated this question numerically, with a high degree of precision, 
and found that our numerical experiments support this
lack of uniqueness conjecture. Thus, we have considered the mKdV equation with initial data given by
$u_c^0$ (see \eqref{num-30}) for various values of $c\le -1$, computed the solution $u_c(x,t)$
up to some time $T>0$,
applied the Miura transform $\mathcal{M}(u_c(x,t))$, and finally compared the solutions obtained.

We have observed a clear dependence of $\mathcal{M}(u_c(x,t))$ as $c\le-1$ varies, see 
Figure~\ref{fig-60}.
This provides strong numerical evidence in support of a non-uniqueness property for the KdV equation
with measure initial data and also a non-trivial test of the robustness of our numerical method:
recall that the initial data \eqref{num-30} are discontinuous functions in $L^2$ only.

Note that these simulations were computed with an accuracy of 30000 spatial points. Due to the 
slow decay of the solution, the computational domain is taken to be 
the interval $[-500,500]$, giving a value of 
$h = 1/30.$ We have also performed the computations with a coarser grid of 5000 points, 
and have found that the (natural) slight variation with $h$ does not affect the overall qualitative
aspect of the solution. In other words, the lack of uniqueness conjecture is strongly supported
by our precise numerical experiments.

Finally, we have verified as well that the result does not depend on the viscosity parameter $\eta$
appearing in \eqref{num-10}. The simulations presented take $\eta = 0.001$, but considering larger values
of $\eta$ (up to $\eta = 0.1$) gives virtually indistinguishable results.

In fact, it is easy to check that the more general family
\be
\label{num-40}
u_{c,\eps}^0(x) = \left\{
\aligned
&\frac{c+\eps}{(c+\eps) x +c}, && x>0,
\\
&\frac1{x+c}, && x<0
\endaligned
\right.
\ee
verifies $\mathcal{M}(u^0_{c,\eps}) = N(c,\eps) \delta_0,$ with $N = (c+\eps-1)/c$. The same 
remarks about uniqueness apply, and so as a last test
we have carried out simulations with $(c,\eps)=(-1/4,1/4)$ and
$(c,\eps)=(-1,-2)$ (for which $N(c,\eps) = 4$), performing the same comparison of the Miura 
transform of the computed solutions. 

For these simulations we have taken a very fine grid of 50000 spatial points, which corresponds to 
$h=0.02.$ The viscosity parameter is $\eta=0.001$. In Figure~\ref{fig-70} we plot the 
Miura transform of the solution for two different values of $(c,\eps)$, with $T=10$, and 10000 and 50000 spatial points for each value of $(c,\eps)$.

Again, some variation with $h$ is observed 
for the same values of $(c,\eps)$, which is natural since the scheme includes dissipation. 
But the main thing to note are the appearance	of two distinct solutions, one for each set of values 
of the pair $(c,\eps),$ clearly apparent in
Figure~\ref{fig-70}. The same distinction between the two solutions is also apparent for 
intermediate values of the number of grid points, whose solutions are seen to lie smoothly between
the ones presented here.

We can therefore conclude that our numerical experiments strongly indicate lack of 
uniqueness for the Cauchy problem
for the KdV equation with a measure initial data.

\begin{figure}
\includegraphics[width=\linewidth,keepaspectratio=false]{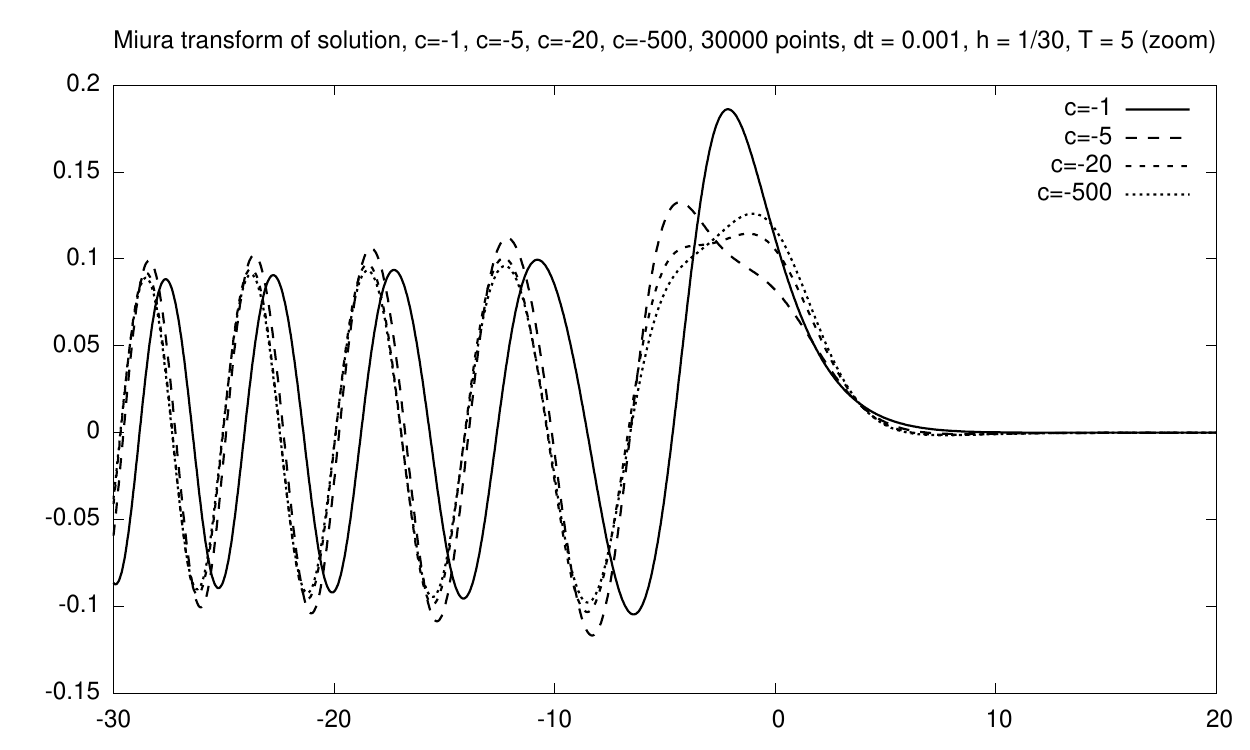}
\caption{Miura transform of solution for various values of $c$. $T=5$, 30000 spatial points, $\tau=0.001$.}
\label{fig-60}
\end{figure}

\begin{figure}
\includegraphics[width=\linewidth,keepaspectratio=false]{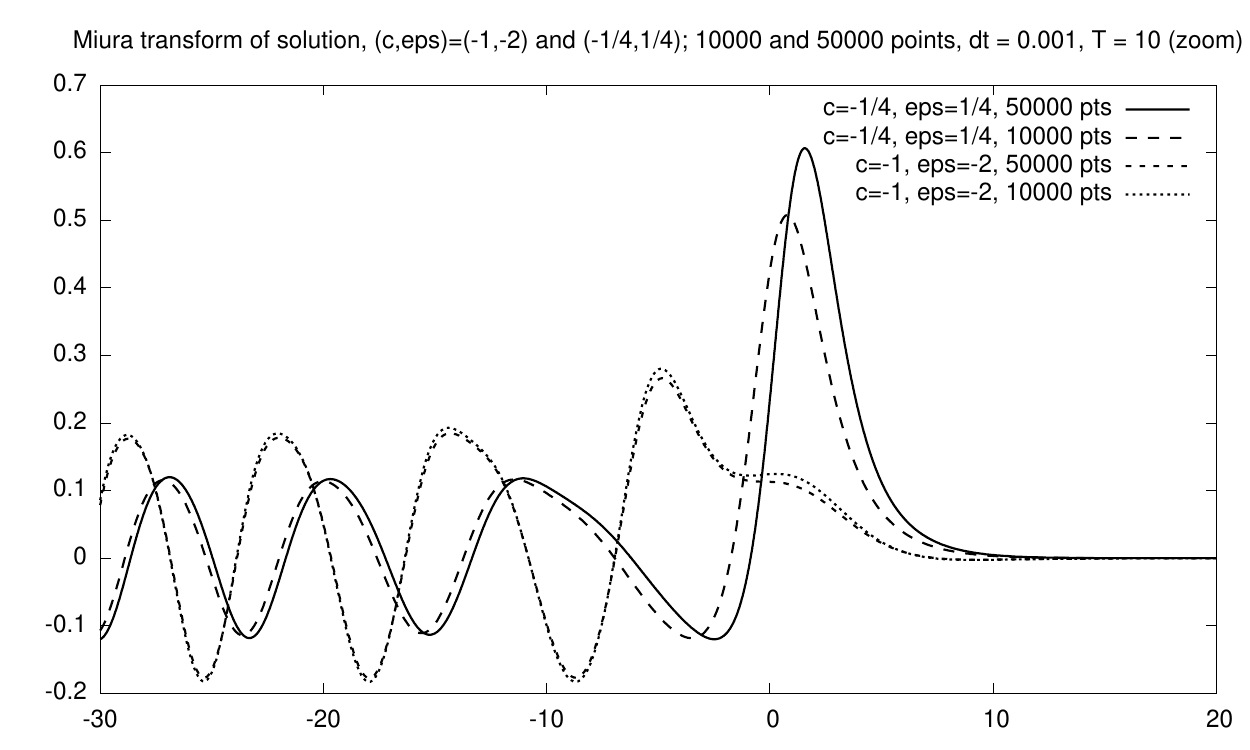}
\caption{Miura transform of solution for various values of $c,\eps$. $T=10$, 10000 and 50000 spatial points. $\tau=0.001$.}
\label{fig-70}
\end{figure}

\section*{Acknowledgements}
The authors were partially supported by the Portuguese Foundation for Science and Technology (FCT)
through the grant PTDC/MAT/110613/2009.
PA was supported by the Portuguese Foundation for Science and Technology (FCT) through a 
\emph{Ci\^encia~2008} fellowship. 



\begin{thebibliography}{10} 

\bibitem{BonaScott} J.L. Bona, L.R. Scott,
Solutions of the Korteweg-de Vries equation in fractional order Sobolev spaces. 
\emph{Duke Math. J.} \textbf{43} (1976), no. 1, 87--99.

\bibitem{BonaSmith} J.L. Bona, R. Smith,
The initial-value problem for the Korteweg-de Vries equation. 
\emph{Philos. Trans. Roy. Soc. London Ser. A} \textbf{278} (1975), no. 1287, 555--601.

\bibitem{Bourgain} J. Bourgain,
Fourier transform restriction phenomena for certain lattice subsets and applications to nonlinear evolution equations.
\emph{Geom. Funct. Anal.} \textbf{3} (1993), no. 3, 107--156, 209--262.

\bibitem{Cohen1} A. Cohen,
Existence and regularity for solutions of the Korteweg-de\thinspace Vries equation. 
\emph{Arch. Rational Mech. Anal.} \textbf{71} (1979), no. 2, 143--175.

\bibitem{Cohen2} A. Cohen
Solutions of the Korteweg-de Vries equation from irregular data. 
\emph{Duke Math. J. } \textbf{45} (1978), no. 1, 149--181.

\bibitem{Cohen3} A. Cohen,
Decay and regularity in the inverse scattering problem. 
\emph{J. Math. Anal. Appl.} \textbf{87} (1982), no. 2, 395--426

\bibitem{MiuraEtc}
{C.S. Gardner, J.M. Greene, M.D. Kruskal, R.M. Miura.} 
{Korteweg-de\thinspace Vries equation and generalization. VI. Methods for exact solution.} \emph{Comm. Pure Appl. Math.} \textbf{27} (1974), 97--133.

\bibitem{GT} J. Ginibre, Y. Tsutsumi,
Uniqueness of solutions for the generalized Korteweg-de Vries equation. 
\emph{SIAM J. Math. Anal.} \textbf{20} (1989), no. 6, 1388--1425,

\bibitem{GTV} J. Ginibre, Y. Tsutsumi, G. Velo,
Existence and uniqueness of solutions for the generalized Korteweg de Vries equation. 
\emph{Math. Z.} \textbf{203} (1990), no. 1, 9--36.

\bibitem{Goda} K. Goda,
On stability of some finite difference schemes for the Korteweg-de Vries equation. 
\emph{J. Phys. Soc. Japan} \textbf{39} (1975), no. 1, 229--236.


\bibitem{Kato1} T. Kato,
On the Korteweg-de\thinspace Vries equation. 
\emph{Manuscripta Math.} \textbf{28} (1979), no. 1-3, 89--99

\bibitem{Kato2} T. Kato,
\emph{On the Cauchy problem for the (generalized) Korteweg-de Vries equation.} Studies in applied mathematics, 93--128, Adv. Math. Suppl. Stud., 8, Academic Press, New York, 1983.

\bibitem{KPV}{C. Kenig, G. Ponce, and L. Vega,}
Well-posedness and scattering results for the generalized Korteweg-de Vries equation via the contraction 
principle. 
\emph{Comm. Pure Appl. Math.} \textbf{46} (1993), no. 4, 527--620

%

\bibitem{Lax} P.D. Lax,
{Integrals of nonlinear equations of evolution and solitary waves. }
\emph{Comm. Pure Appl. Math.} \textbf{21} 1968 467--490.


\bibitem{Lions} J.L. Lions,
\emph{Quelques m\'ethodes de r\'esolution des probl\`emes aux limites non lin\'eaires.} 
Dunod; Gauthier-Villars, Paris 1969

\bibitem{LinaresPonce} F. Linares, G. Ponce,
\emph{Introduction to nonlinear dispersive equations. }
Universitext. Springer, New York, 2009.

\bibitem{Miura} {R.M. Miura,}
{The Korteweg-de Vries equation: a survey of results.} 
\emph{SIAM Rev.} \textbf{18} (1976), no. 3, 412--459

\bibitem{Nixon}M. Nixon,
The discretized generalized Korteweg-de Vries equation with fourth order nonlinearity. 
\emph{J. Comput. Anal. Appl.} \textbf{5} (2003), no. 4, 369--397.

\bibitem{YuSanzSerna} K. Pen-Yu, J.M. Sanz-Serna
Convergence of methods for the numerical solution of the Korteweg-de Vries equation. 
\emph{IMA J. Numer. Anal.} \textbf{1} (1981), no. 2, 215--221

\bibitem{Pazoto}
A. Pazoto, M. Sep\'ulveda, O.V. Villagr\'an, Uniform stabilization of numerical schemes for the critical generalized Korteweg-de Vries equation with damping. 
\emph{Numer. Math.} \textbf{116} (2010), no. 2, 317--356


\bibitem{SautTemam} J.C. Saut, R. Temam,
Remarks on the Korteweg-de Vries equation. 
\emph{Israel J. Math.} \textbf{24} (1976), no. 1, 78--87.

\bibitem{Scott} {A.C. Scott, F.Y. Chu, \& D.W. McLaughlin,} The soliton: A new concept in applied science,
\emph{Proc. IEEE,} \textbf{61} (1973), 1443--1483

\bibitem{Tsutsumi} Y. Tsutsumi,
The Cauchy problem for the Korteweg-de Vries equation with measures as initial data. 
\emph{SIAM J. Math. Anal.} \textbf{20} (1989), no. 3, 582--588.

\bibitem{Zhou} Y. Zhou,
Uniqueness of weak solutions of the KdV equation.
\emph{Internat. Math. Res. Not.} (1997), no. 6 271--283

\end{thebibliography}
\end{document}